\documentclass[10pt]{amsart}
\usepackage{latexsym, amsmath,amssymb}

\usepackage{epsfig}

\newcommand{\newsection}[1]
{\subsection{#1}\setcounter{theorem}{0} \setcounter{equation}{0}
\par\noindent}

\newtheorem{theorem}{Theorem}

\newtheorem{lemma}[theorem]{Lemma}
\newtheorem{corollary}[theorem]{Corollary}
\newtheorem{proposition}[theorem]{Proposition}

\newcommand{\cal}{\mathcal}

\newcommand{\R}{{\mathbb R}}

\newcommand{\ep}{\varepsilon}

\newcommand{\hf}{\frac 12}
\newcommand{\cd}{\,\cdot\,}

\newcommand{\Kob}{{\cal K}}

\begin{document}

\title[Strichartz estimates for Dirichlet-wave equations in two dimensions]
{Strichartz estimates for Dirichlet-wave  equations  in two dimensions with
applications}
\thanks{The authors were supported in part by the NSF. The third author was supported in part by NSFC 10871175 and 10911120383.}

\author{Hart F. Smith}
\address{Department of Mathematics, University of Washington, Seattle, WA 98195}
\author{Christopher D. Sogge}
\address{Department of Mathematics,  Johns Hopkins University,
Baltimore, MD 21218}
\author{Chengbo Wang}
\address{Department of Mathematics,  Johns Hopkins University,
Baltimore, MD 21218}

\begin{abstract}
We establish the Strauss conjecture for nontrapping obstacles
when the spatial dimension $n$ is two.  As pointed out in \cite{HMSSZ}
this case is more subtle than $n=3$ or $4$ due to the fact that
the arguments of the first two authors \cite{SmSo00},
Burq \cite{B} and Metcalfe \cite{M} showing that local Strichartz estimates
for obstacles
imply global
ones require that the Sobolev
index, $\gamma$, equal $1/2$ when $n=2$.   We overcome this difficulty by interpolating
between energy estimates ($\gamma =0$) and ones for $\gamma=\frac12$
that are generalizations of Minkowski space estimates of Fang
and the third author \cite{FaWa2}, \cite{FaWa}, the second author \cite{So08}
and Sterbenz \cite{St05}.
\end{abstract}

\maketitle

\newsection{Introduction}

In a recent series of papers, \cite{DMSZ}, \cite{HMSSZ}, techniques have been
developed to prove general Strichartz estimates for wave equations outside of
nontrapping obstacles.  These papers relied on ideas that were used to prove
the more standard $L^q_tL^r_x$ Strichartz estimates for obstacles in  \cite{B},
\cite{M} and \cite{SmSo00}.  As was shown in \cite{HMSSZ}, though, a limitation
arises in the proof which is only relevant when the spatial dimension, $n$,
equals two.  This is that the $TT^*$ arguments involving the Christ-Kiselev
lemma \cite{ChKi01} {\it a priori} require that the Sobolev regularity for the
data in the homogeneous estimates be equal to $\hf$ when $n=2$, with similar
restrictions on the estimates for the inhomogeneous wave equation.

As we shall see in this paper, even though we can only directly prove
Strichartz estimates involving Sobolev regularity of $\gamma=\hf$, for some
applications if we interpolate with
trivial (energy) estimates, this is enough.  In particular, we shall be able
to establish
the Strauss conjecture for obstacles when $n=2$.
Specifically, if $\Kob \subset \mathbb{R}^2$ is a compact nontrapping obstacle
with smooth boundary, then we shall be able
to show that there are global small-amplitude solutions of the equation
\begin{equation}\label{1.1}
\begin{cases}
\square u(t,x)=F_p(u(t,x)), \quad (t,x)\in \mathbb{R}_+\times
\mathbb{R}^2 \backslash \Kob
\\
u(t,x)=0, \quad x\in \partial\Kob
\\
u|_{t=0}=f, \quad \partial_t u|_{t=0}=g
\end{cases}
\end{equation}
provided that
\begin{equation}\label{1.2}
|F_p(u)|+|u|\, |F'_p(u)| \lesssim |u|^p, \quad \text{for } \, |u|\le 1,
\end{equation}
and $p$ is larger than the critical exponent for \eqref{1.1} when $n=2$, which
is $p_c=(3+\sqrt{17})/2$, or equivalently
\begin{equation}\label{1.3}
p^2-3p-2>0, \quad \text{and } \, \, p>0.
\end{equation}
Fortunately, the critical family of estimates that we require for proving
bounds for this
equation involve $\gamma=\hf$ (see Figure \ref{figu1} below).  All the other estimates,
including
the ones we shall use, come
from interpolating between these and energy estimates.

Let us state our existence results for \eqref{1.1} with more precision.
We first introduce some notation.  We will denote
$$
\Omega = {\mathbb R}^2 \backslash \Kob
$$
and let
$\dot H^\gamma (\Omega)$ be the homogeneous Sobolev space of order $\gamma$
on $\Omega$, with norm
$$\|f\|_{\dot H^\gamma(\Omega)}= \|\, (\sqrt{-\Delta_D})^\gamma f\|_{L^2(\Omega)},$$
with $\Delta_D$ the Dirichlet-Laplacian in $\Omega$.  If
$0\le \gamma <\hf$
then for $f\in C^\infty({\Omega})$ we have
$$\|f\|_{\dot H^\gamma(\Omega)} \approx \|\tilde f\|_{\dot H^\gamma({\mathbb R}^2)},$$
if $\tilde f(x)=f(x)$, $x\in\Omega$ and $\tilde f(x)=0$,
$x\in\Kob$.
Here $\dot H^\gamma({\mathbb R}^2)$ denotes the homogeneous Sobolev space with
norm
$$\|g\|^2_{\dot H^\gamma({\mathbb R}^2)}=(2\pi)^{-2}\int_{{\mathbb R}^2} |\, |\xi|^\gamma
\Hat g(\xi)|^2\, d\xi,$$
with $\Hat g$ denoting the the Fourier transform of $g$. See also the
introduction of \cite{HMSSZ} for a discussion of the space
$\dot{H}^\gamma(\Omega)$.

If we also let $\partial_j = \partial_{x_j}$, $j=1,2$ and
\begin{equation}\label{1.4}
\{Z\}=\{\partial_1, \partial_2, \, x_1\partial_2-x_2\partial_1\, \},
\end{equation}
then we can state our existence theorem for \eqref{1.1}. The norm used in \eqref{1.5}
is certainly not the best possible; see the remarks following Corollary \ref{coro-3.3}.

\begin{theorem} \label{thm1.1} Let $n=2$ and $\Kob$, $\Omega$ be as above.
If $p_c<p<5$, then there is an
$\varepsilon_0=\varepsilon_0(p,\Omega)>0$  such that \eqref{1.1} has a global
solution satisfying
$$ Z^\alpha u(t, \cd)\in
L^{p-1}(\Omega)\,,\qquad|\alpha|\le 1$$
provided that the initial data
$(f,g)=(u|_{t=0}, \partial_tu|_{t=0})$ satisfies
$f|_{\partial\Omega}=0$
and
\begin{equation}\label{1.5}
\sum_{|\alpha|\le 2} \| Z^\alpha f\|_{ L^{q_p}(\Omega)}
+\sum_{|\alpha|\le 1}\|Z^\alpha g\|_{L^{q_p}(\Omega)}<\varepsilon,
\quad 0<\varepsilon<\varepsilon_0,
\end{equation}
with $\frac 1{q_p}=\frac 1{p-1}+\frac 12$.
If $p\ge5$, then there is a global solution of \eqref{1.1} if
\eqref{1.5} holds with $q= q_{\tilde p}$, for some $\tilde p \in (p_c,5)$.
\end{theorem}

Note that by Sobolev embedding
\begin{multline}\label{sob}
\sum_{|\alpha|\le1}\bigl(\|Z^\alpha f\|_{\dot H^{\gamma_p}(\Omega)}+
\|Z^\alpha g\|_{\dot H^{\gamma_p-1}(\Omega)}\bigr)
\\
\le C_p\Bigl(\, \sum_{|\alpha|\le2}\|Z^\alpha f\|_{L^{q_p}(\Omega)}+
\sum_{|\alpha|\le1}\|Z^\alpha g\|_{L^{q_p}(\Omega)}\, \Bigr),
\end{multline}
where $q_p$ is as above and
$$\gamma_p=1-\frac2{p-1}$$
%
is the scaling exponent for the equation
$\square u =|u|^p$ in two dimensions (see, e.g. \cite{So08}).  In earlier works
(\cite{DMSZ}, \cite{HMSSZ}) the smallness assumption on the data was based
on the size of the $\dot H^{\gamma_p}(\Omega)\times \dot H^{\gamma_p-1}(\Omega)$
norm of derivatives of $(f,g)$ (see also \eqref{19} below).  For technical reasons, we are led to making the
somewhat stronger assumption \eqref{1.5} involving the $L^{q_p}$-norms, but this too
is natural.

To prove Theorem~\ref{thm1.1} we shall require certain Strichartz estimates
in $\Omega$.  We shall postpone formulating them until they are needed in \S 3,
but they are related to the following 2-dimensional Minkowski space estimates,
which involve the angular mixed-norm spaces
$$
\|f\|_{L^r_{|x|}L^2_\theta({\mathbb R}^2)}=
\left(\, \int_0^\infty \, \Bigl( \, \int_0^{2\pi}
|f(\rho(\cos \theta, \sin \theta))|^2 \, d\theta
\, \Bigr)^{r/2}\, \rho \, d\rho\, \right)^{1/r}\,.
$$

\begin{proposition}\label{prop1.2}  Let $P=\sqrt{-\Delta}$ in ${\mathbb R}^2$.
Assume that
$(q,r)\ne (\infty,\infty)$
\begin{equation}\label{1.6}
q,r >2 \quad \text{and } \, \, \frac 1q < \frac 12 - \frac 1r\,,
\end{equation}
or $(q,r)=(\infty,2)$.
Then
\begin{equation}\label{1.7}
\bigl\|\, e^{-itP}g \,  \bigr\|_{L^q_tL^r_{|x|}L^2_\theta({\mathbb R}\times {\mathbb R}^2)}
\le C_{q,r}\|g\|_{\dot H^{\gamma}({\mathbb R}^2)}\,,
\quad \gamma= 2(\tfrac 12-\tfrac 1r)-\tfrac 1q\,.
\end{equation}
\end{proposition}

See the following figure for the range of exponents in \eqref{1.7}:
\begin{figure*}[h]
\centering
\includegraphics[width=0.9\textwidth]{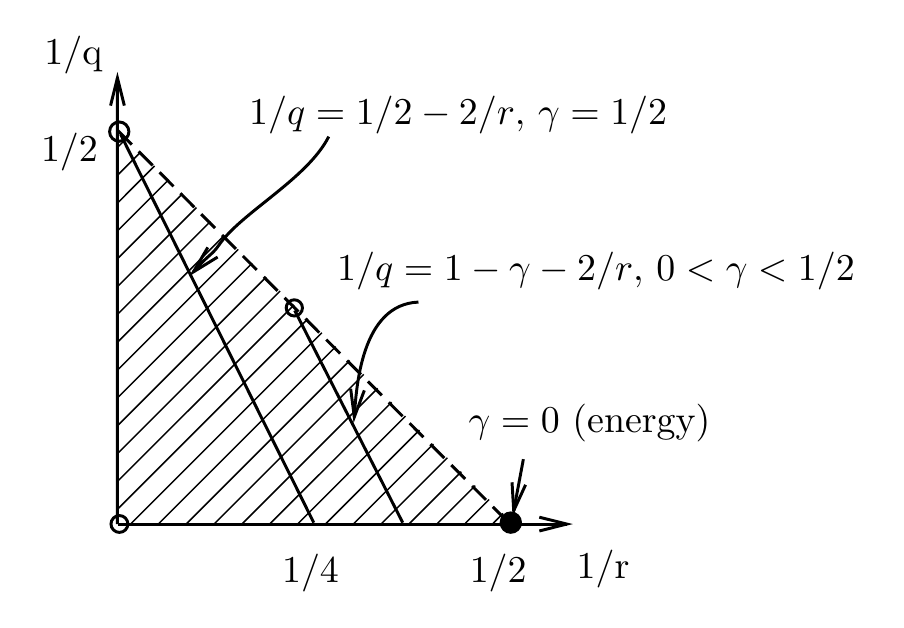}
\caption{Minkowski space exponents}\label{figu1}
\end{figure*}

We mention that Sterbenz \cite{St05} proved related estimates where $L^2_\theta$
is replaced by $L^r_\theta$ (with norms of different regularity on the right).
Related results
are also due to Fang and Wang \cite{FaWa2}, \cite{FaWa} and Sogge \cite{So08} (for $n=3$).
Since the proof of
\eqref{1.7} is simple, we shall present it in \S 2 for the sake of
completeness. It can be adapted to give a slightly different proof of the
corresponding results in \cite{St05} and \cite{FaWa}.

This paper is organized as follows.  In the next section, we shall
prove Proposition~\ref{prop1.2}. We shall also show how it can be
used to give a simple proof of Glassey's theorem \cite{G} which says
that the Strauss conjecture holds for ${\mathbb R}_+\times {\mathbb
R}^2$, since this will serve as a model for the more technical
arguments that are needed to establish Theorem~\ref{thm1.1}. In the
final section, we shall formulate and prove the variants of
\eqref{1.7} that we require and then present the proof of this
theorem.

\newsection{Estimates for ${\mathbb R}_+\times {\mathbb R}^2$ and
Glassey's Theorem}

We shall first prove Proposition~\ref{prop1.2} and then give the simple
argument showing how it can be used to prove Glassey's Theorem that in
${\mathbb R}_+\times {\mathbb R}^2$ there is small amplitude global existence
for $\square u = |u|^p$ when $p>p_c=(3+\sqrt{17})/2$.

The main step in the proof of \eqref{1.7} will be to show that
\begin{equation}\label{a}
\|e^{-itP}f\|_{L^q_tL^\infty_{|x|}L^2_\theta({\mathbb R}\times {\mathbb R}^2)}
\le C_q \|f\|_{L^2({\mathbb R}^2)}, \;\;
\text{if }q>2, \text{ and }
\Hat f(\xi)=0 \text{ if } |\xi|\notin [\tfrac 12,1].
\end{equation}
Since Hardy-Littlewood-Sobolev estimates give
$\dot H^{1-\frac 2r}({\mathbb R}^2)\subset
L^r({\mathbb R}^2)$, $2\le r<\infty$, we also clearly have
\begin{equation}\label{energy}
\bigl\|\, e^{-itP}f\, \bigr\|_{L^\infty_tL^r_{|x|}L^2_\theta}\le C_r\|f\|_{\dot H^{1-\frac 2r}},
\end{equation}
since $e^{-itP}$ is a unitary operator on $\dot H^\gamma$.  The estimates
\eqref{a} and \eqref{energy}
say that we have the estimates described in Figure \ref{figu1} that
respectively correspond to the
(open) vertical
and (half open) horizontal segments.
By interpolation we conclude that,
if $q,r>2$ and $\tfrac 1q<\tfrac 12-\tfrac 1r$, then
$$\|e^{-itP}f\|_{L^q_tL^r_{|x|}L^2_\theta}\le C_{q,r}\|f\|_{L^2({\mathbb R}^2)}, \quad
\text{if } \, \Hat f(\xi)=0, \, \, |\xi|\notin [\tfrac 12,1].$$
By scaling and Littlewood-Paley theory, we obtain from this that if we
remove the support assumptions on the Fourier transform,
then for $q$ and $r$ as above, and $(q,r)\neq(\infty,\infty)$,
\begin{equation}\label{b}
\|e^{-itP}g\|_{L^q_tL^r_{|x|}L^2_\theta({\mathbb R}\times {\mathbb R}^2)}
\le C_{q,r}\|g\|_{\dot H^{1-\frac 2r-\frac 1q}({\mathbb R}^2)},
\end{equation}
which is the inequality in Proposition~\ref{prop1.2}.

Let us turn to the proof of \eqref{a}.  By the support assumptions for
$\Hat f$ we have that
\begin{equation}\label{c}
\|f\|^2_{L^2({\mathbb R}^2)}\approx \int_0^\infty \int_0^{2\pi}|\Hat f(\rho(\cos\omega,\sin\omega))|^2
d\omega d\rho.
\end{equation}
We expand the angular part of $\Hat f$ using Fourier series and find that if
$\xi = \rho(\cos\omega,\sin\omega)$, then there are coefficients
$c_k(\rho)$ which vanish
for $\rho\notin [\tfrac 12,1]$, so that
$$\Hat f(\xi)=\sum_k c_k(\rho)\,e^{ik\omega}\,.$$
By \eqref{c} and Plancherel's theorem for $S^1$ and ${\mathbb R}$, we have
\begin{equation}\label{d}
\|f\|^2_{L^2({\mathbb R}^2)}
\approx \sum_k\int_{\mathbb R}|c_k(\rho)|^2\, d\rho \approx \sum_k\int_{{\mathbb R}}
|\Hat c_k(s)|^2 \, ds,
\end{equation}
where $\Hat c_k(s)$, $s\in {\mathbb R}$,
denotes the one-dimensional Fourier transform of $c_k(\rho)$.
Recall  that (see Stein and Weiss \cite{SW} p. 137)
\begin{equation}\label{e}
f(r(\cos\omega,\sin\omega))=(2\pi)^{-1}\sum_k
\Bigl(\, i^k \int_0^\infty J_{k}(r\rho)\,c_k(\rho)\,\rho \, d\rho\,
\Bigr) e^{ik\omega},
\end{equation}
where $J_k$, $k\in {\mathbb Z}$, is the $k$-th Bessel function, defined by
\begin{equation}\label{f}
J_k(y)=\frac{(-i)^k}{2\pi}\int_0^{2\pi} e^{iy\cos\theta -ik\theta}\, d\theta.
\end{equation}
By \eqref{e} and the support properties of the $c_k$,
if we fix
$\beta\in C^\infty_0({\mathbb R})$ satisfying $\beta(\tau)=1$ for
$\tau\in[\frac12,1]$ and $\beta(\tau)=0$ for $\tau\notin [\frac14,2]$,
then with $\alpha(\rho)=\rho\,\beta(\rho)
\in {\mathcal S}({\mathbb R})$, we have
\begin{align*}
\bigl(e^{-itP}f\bigr)&(r(\cos\omega,\sin\omega))
\\
&=
(2\pi)^{-1}\sum_k\Bigl(\, i^k \int_0^\infty J_{k}(r\rho)\,e^{-it\rho}\,c_k(\rho)\,
\beta(\rho)\,\rho\, d\rho\, \Bigr)e^{ik\omega}
\\
&=(2\pi)^{-2}\sum_k\Bigl(\, i^k\int_0^\infty \int_{-\infty}^\infty J_{k}(r\rho)\,
e^{i\rho(s-t)}\,\Hat c_k(s)\,\alpha(\rho)\, ds\, d\rho\, \Bigr)e^{ik\omega}
\\
&=(2\pi)^{-3}\sum_k\Bigl(\,
\int_0^\infty\int_{-\infty}^\infty \int_0^{2\pi}
e^{i\rho r\cos\theta}e^{-ik\theta}e^{i\rho(s-t)}\,\Hat c_k(s)\,
\alpha(\rho)\, d\theta\, ds \,d\rho\, \Bigr)e^{ik\omega}
\\
&=(2\pi)^{-3}\sum_k\Bigl(\,  \int_{-\infty}^\infty \int_0^{2\pi}
e^{-ik\theta}\Hat \alpha\bigl((t-s)-r\cos\theta \bigr)\, \Hat c_k(s) \,
d\theta \, ds \, \Bigr)e^{ik\omega}.
\end{align*}
As a result, we have that for any $r\ge0$,
\begin{multline}\label{g}
\int_0^{2\pi}\Bigl| \, \bigl(e^{-itP}f\bigr)
(r(\cos\omega,\sin\omega))\, \Bigr|^2 \, d\omega
\\
=(2\pi)^{-5} \sum_k\,\Bigl|\, \int_{-\infty}^\infty \int_0^{2\pi} e^{-ik\theta}\,
\Hat \alpha\bigl((t-s)-r\cos\theta\bigr)\, \Hat c_k(s) \,d\theta\,ds\, \Bigr|^2.
\end{multline}

To estimate the right side we shall use the following.

\begin{lemma}\label{mainest}
  Let $\alpha\in {\mathcal S}({\mathbb R})$ and $N\in {\mathbb N}$ be fixed.  Then there is a uniform constant $C$, which is independent of $m\in {\mathbb R}$ and $r\ge0$, so that the following inequalities hold.  First,
\begin{equation}\label{1}
\int_0^{2\pi} |\alpha(m-r\cos\theta)|\, d\theta \le C \langle \, m\, \rangle^{-N}, \quad
\text{if } \, \, 0\le r\le 1, \, \, \text{or }\, \, |m|\ge 2r.
\end{equation}
If $r>1$ and $|m|\le 2r$ then
\begin{equation}\label{2}
\int_0^{2\pi}|\alpha(m-r\cos\theta)|\, d\theta
\le C\Bigl(\, r^{-1}+r^{-\frac12}\langle \, r-|m|\, \rangle^{-\frac12}\, \Bigr).
\end{equation}
Consequently, if $\delta>0$, there is a constant $A_\delta$, which is independent of $t\in {\mathbb R}$ and $r\ge 0$ so that
\begin{equation}\label{3}
\int_{-\infty}^\infty\left( \, \int_0^{2\pi} \langle \, t-s \, \rangle^{\frac12-\delta}\,
|\alpha((t-s)-r\cos\theta)|\, d\theta \, \right)^2 \, ds \le A_\delta.
\end{equation}
\end{lemma}

If we apply \eqref{3} and \eqref{g} along with the Schwarz inequality, we
conclude that if $f$ is as in \eqref{a}, then for $\delta>0$
$$\Bigl\| \, e^{-itP}f\, \Bigr\|_{L^\infty_{|x|}L^2_\theta}^2
\le B_\delta \sum_k\int_{-\infty}^\infty \, \bigl|\, \langle\, t-s\,
\rangle^{-\hf+\delta}
\Hat c_k(s)\, \bigr|^2 \, ds,$$
which, by Minkowski's inequality and \eqref{d}, in turn yields \eqref{a}.

\noindent {\bf Proof of Lemma~\ref{mainest}:}
We first realize that inequalities \eqref{1} and \eqref{2} clearly imply \eqref{3}.  Also, \eqref{1} is trivial since $\alpha\in {\mathcal S}$.  Therefore, we just need to to prove \eqref{2}.  To do so, it suffices to show that
\begin{equation}\label{4}
\int_0^{\pi/4}|\alpha(m-r\cos\theta)|\, d\theta +\int_{\pi-\pi/4}^\pi |\alpha(m-r\cos\theta)|\, d\theta
\le C r^{-\frac12}\langle \, r-|m|\, \rangle^{-\frac12},
\end{equation}
and also
\begin{equation}\label{5}
\int_{\pi/4}^{\pi-\pi/4}|\alpha(m-r\cos\theta)| \, d\theta \le Cr^{-1}.
\end{equation}

In order to prove \eqref{4}, it suffices to prove that the first integral is controlled by the right side.  For if we apply this estimate to the function $\alpha(-s)$, we then see that the second integral satisfies the same bounds.  We can estimate the first integral if we make the substitution
$u=1-\cos\theta$, in which case, we see that it equals
\begin{align*}
\int_0^{1-1/\sqrt2}|\alpha((m-r)+ru)|\, \frac{du}{\sqrt{2u-u^2}}
&\le \int_0^{1-1/\sqrt2}|\alpha((m-r)+ru)|\, \frac{du}{\sqrt u}
\\
&\le Cr^{-\frac12}\int_0^\infty |\alpha((m-r)+u)|\, \frac{du}{\sqrt u}
\\
&\le C'r^{-\frac12}\, \langle \, r-|m| \, \rangle^{-\frac12},
\end{align*}
as desired, which completes the proof of \eqref{4}.

To prove \eqref{5} we just make the change of variables $u=r\cos\theta$ and note that $|du/d\theta| \approx r$ on the region of integration,
which leads to the inequality as $\alpha\in {\mathcal S}$.  \qed

\medskip

We conclude this section by showing how Proposition~\ref{prop1.2}
implies estimates that can be used to prove
Glassey's~\cite{G} existence theorem for $\square u =|u|^p$ when $n=2$.
Specifically, if $u$
solves the wave equation for ${\mathbb R}\times {\mathbb R}^2$,
\begin{equation}\label{14}
\begin{cases}
\square u = F
\\
u|_{t=0}=f, \quad \partial_tu|_{t=0}=g,
\end{cases}
\end{equation}
then
\begin{equation}
\label{15}
\|u\|_{L^q_tL^r_{|x|}L^2_\theta}+\|u\|_{L^\infty_t\dot H^\gamma}\lesssim
\|f\|_{\dot H^\gamma}+\|g\|_{\dot H^{\gamma-1}}+\|F\|_{L_t^{\tilde q'}L^{\tilde r'}_{|x|}L^2_\theta},
\end{equation}
assuming that $q,r, \tilde q, \tilde r>2$ with
$(q,r), (\tilde q,\tilde r) \ne (\infty, \infty)\,,$
$\frac 1q<\hf-\frac 1r$, $\frac 1{\tilde q}< \hf - \frac 1{\tilde r}$, and
\begin{equation}\label{16}
\gamma=1-\frac 2r-\frac 1q\,, \quad \text{and } \,
1-\gamma = 1-\frac 2{\tilde r}-\frac 1{\tilde q}\,.
\end{equation}
In \eqref{15}, $\tilde q'$ and $\tilde r'$ denote the exponents which are
conjugate to
$\tilde q$ and $\tilde r$, respectively, and also, here and in what follows,
the space-time norms are taken over ${\mathbb R}_+\times {\mathbb R}^2$.
 Clearly, \eqref{15} follows from \eqref{1.7} and
energy estimates if the forcing term, $F$, in \eqref{14} vanishes.  Since we
are
assuming \eqref{16} and since $\tilde q'<q$, the estimates for the
inhomogeneous wave
equation follow from an application of the Christ-Kiselev lemma \cite{ChKi01}
(cf. \cite{So08},
pp. 136--141).

If $\{Z\}$ are the operators in \eqref{1.4}, then since they commute with
$\square$, \eqref{15}
implies that
\begin{multline}\label{17}
\sum_{|\alpha|\le1}
\Bigl(\,
\|Z^\alpha u\|_{L^q_tL^r_{|x|}L^2_\theta}+\|Z^\alpha u\|_{L^\infty_t \dot H^\gamma}\, \Bigr)
\\
\lesssim
\sum_{|\alpha|\le1}
\Bigl(\, \|Z^\alpha f\|_{\dot H^\gamma}+\|Z^\alpha g\|_{\dot H^{\gamma-1}}
+\|Z^\alpha F\|_{L^{\tilde q'}_tL^{\tilde r'}_{|x|}L^2_\theta}\, \Bigr),
\end{multline}
with $q,r,\tilde q', \tilde r'$ and $\gamma$ as above.  Let us now present
the simple argument
showing that this estimate implies that there are global solutions of the
equation
\begin{equation}\label{18}
\begin{cases}
\square u(t,x) = F_p(u(t,x)), \quad (t,x)\in {\mathbb R}_+\times {\mathbb R}^2
\\
u|_{t=0}=f, \quad \partial_t u|_{t=0}=g,
\end{cases}
\end{equation}
if $F_p$ is as in \eqref{1.2}, with $p$ as in \eqref{1.3}, assuming
that (when $p_c<p<5$) the initial data satisfies
\begin{equation}\label{19}
\sum_{|\alpha|\le 1}\Bigl(\, \|Z^\alpha f\|_{\dot H^{\gamma_p}}+\|Z^\alpha g\|_{\dot
H^{\gamma_p-1}}\, \Bigr)<\varepsilon, \quad \gamma_p=1-\tfrac2{p-1},
\end{equation}
with $\varepsilon=\varepsilon(p)$ sufficiently
small.

We first consider the subconformal range where $\frac{3+\sqrt{17}}2=p_c<p<5$.
This range easily lends itself to the special case of \eqref{17},
which says that, for such $p$,
\begin{multline}\label{20}
\sum_{|\alpha|\le 1}\Bigl(\,
\|Z^\alpha u\|_{L^{\frac{(p-1)p}{2}}_tL^p_{|x|}L^2_\theta}+\|Z^\alpha u\|_{L^\infty_t\dot
H^{\gamma_p}}\Bigr)
\\
\lesssim \sum_{|\alpha|\le 1}\Bigl(\,
\|Z^\alpha f\|_{\dot H^{\gamma_p}}+\|Z^\alpha g\|_{\dot H^{\gamma_p-1}}
+\|Z^\alpha F\|_{L^{\frac{p-1}{2}}_t L^1_{|x|}L^2_\theta}\,\Bigr)\,.
\end{multline}
The temporary assumption that $p<5$ is needed to ensure that $(p-1)/2<2$,
and, therefore,
$[(p-1)/2]'>2$, which is the first part of the assumptions for \eqref{15}.
The more serious
assumption that $p>p_c$, which is \eqref{1.3},
is equivalent to the second part of
\eqref{1.6} for the exponents on the left side of \eqref{20}. That is,
for $p>0$,
$$\frac2{p(p-1)}<\frac12-\frac1p \, \iff \, p>p_c.$$

Using \eqref{20}, we shall show that we can solve \eqref{18} by an
iteration argument
for $p_c<p<5$, provided that $\varepsilon>0$ in \eqref{19} is small.
To be more specific,
we shall let $u_0$ solve the Cauchy problem \eqref{14} with $F\equiv 0$.
We then iteratively
define $u_k$, $k\ge1$, by solving
\begin{equation}
\begin{cases}
\square u_k (t,x)=F_p(u_{k-1}(t,x)), \quad
(t,x)\in {\mathbb R}_+\times {\mathbb R}^2
\\
u_k|_{t=0}=f ,\quad \partial_t u_k|_{t=0}=g\,.
\end{cases}
\end{equation}
Our aim is to show that if $\varepsilon>0$ in \eqref{19} is small enough, then
$$
M_k=\sum_{|\alpha|\le 1} \Bigl(\,
\|Z^\alpha u_k\|_{L^{\frac{(p-1)p}2}_tL^p_{|x|}L^2_\theta}+
\|Z^\alpha u_k\|_{L^\infty_t\dot H^{\gamma_p}}\, \Bigr)
$$
must also be small.

For $k=0$, it follows from \eqref{20} that $M_0\le C_0\,\varepsilon$, with
$C_0$ a fixed constant.  Clearly, \eqref{20} also yields
that for $k=1,2,3,\cdots$
$$
M_k\le C_0\,\varepsilon +
C_0\sum_{|\alpha|\le 1}\|Z^\alpha F_p(u_{k-1})\|_{L^{\frac{p-1}2}_t
L^1_{|x|}L^2_\theta}\,.
$$
To control the last term, we note that our assumption \eqref{1.2} on $F_p$
implies that
$$
\sum_{|\alpha|\le 1}|Z^\alpha F_p(v)| \lesssim
|v|^{p-1}\sum_{|\alpha|\le 1}|Z^\alpha v|,
\quad \text{if } \, |v|\le 1\,.
$$
Since $\partial_\theta=x_1\partial_2-x_2\partial_1\in \{Z\}$, we have
$$
\|v(|x|\, \cdot \, )\|_{L^\infty_\theta}\lesssim
\sum_{|\alpha|\le1}\|Z^\alpha v(|x|\, \cdot \, )\|_{L^2_\theta},
$$
and since $0<\gamma_p<1$ and $\partial_j\in \{Z\}$, $j=1,2$,
Sobolev estimates imply that
$$
\|v\|_{L^\infty({\mathbb R}^2)}\lesssim
\sum_{|\alpha|\le 1}\|Z^\alpha v\|_{\dot H^{\gamma_p}({\mathbb R}^2)}
$$
so that \eqref{1.2} applies in our case.
Combining the above inequalities gives
$$
M_k\le C_0\,\varepsilon +C_1 C_0 M^p_{k-1}\,,
$$
for some uniform constant $C_1$.  Since $M_0\le C_0\,\varepsilon$,
we deduce from this that, if $\varepsilon>0$ is sufficiently small, then
\begin{equation}\label{21}
M_k\le 2 C_0\,\varepsilon, \quad k=1,2,3,\dots
\end{equation}

To finish the proof of the existence results for $p_c<p<5$,
it suffices to show that
$$
A_k=\|u_k-u_{k-1}\|_{L^{\frac{(p-1)p}2}_tL^p_{|x|}L^2_\theta}
$$
tends geometrically to zero as $k\to \infty$.  Since
$|F_p(w)-F_p(v)|\lesssim
|v-w|\cdot (|v|^{p-1}+|w|^{p-1})$ when $|v|, \, |w|\le 1$,
the proof of \eqref{21} can
be adapted to show that, for small $\varepsilon>0$, there is a uniform
constant $C$ so that
$$A_k\le CA_{k-1}\bigl(M_{k-1}+M_{k-2}\bigr)^{p-1},$$
which, by \eqref{21}, implies that $A_k\le \tfrac12 A_{k-1}$ for small
$\varepsilon>0$.  Since $A_1$
is finite, the claim follows, which finishes the proof of the existence results
for $p_c<p<5$.

As we noted above, we cannot directly get the existence results from
\eqref{20} if
$p\ge 5$.  However, since the assumptions \eqref{1.2} on $F_p$ become weaker
with increasing $p$, the above argument yields existence
results for this case as well.

\newsection{The Strauss conjecture for nontrapping obstacles in 2-dimensions}

The goal of this section is to show that we can solve the semilinear
Dirichlet-wave equation
\eqref{1.1} for small data when $\Kob\subset {\mathbb R}^2$ is a
nontrapping obstacle and,
as in \eqref{1.3}, $p>p_c=\frac{3+\sqrt{17}}2$.  The main step will be to
find a suitable
variant of the Minkowski space estimate \eqref{1.7} which is valid for
solutions of the linear
Dirichlet-wave equation
\begin{equation}\label{3.1}
\begin{cases}
\square u(t,x)=F(t,x), \quad (t,x)\in {\mathbb R}_+\times \Omega
\\
u(t,x)=0, \quad (t,x)\in {\mathbb R}_+\times \partial\Omega
\\
u|_{t=0}=f, \, \, \partial_t u|_{t=0}=g,
\end{cases}
\end{equation}
where, as before, $\Omega={\mathbb R}^2\backslash \Kob$.  As
previously noted, we are in luck because the crucial estimates for
\eqref{1.7} involve Sobolev regularity of $\gamma=\hf$, which
is the regularity necessary for $n=2$ to use the techniques of
\cite{HMSSZ}, \cite{B}, \cite{M} and \cite{SmSo00}, to show that
local in time Strichartz estimates for $\Omega$, coupled with
global in time estimates for ${\mathbb R}^2$, imply global in time estimates
for $\Omega$.  After we obtain these estimates for $\gamma=\hf$, we
shall be able to obtain a family of estimates corresponding to other
$\gamma$ by interpolating with energy estimates.  The range of
exponents will be slightly smaller than in the previous
section, in that we shall not be able to obtain indices on
the open vertical line segment in Figure \ref{figu1} connecting
$(\frac 1r,\frac 1q)=(0,0)$ and $(0,\hf)$ (see Figure \ref{figue2} below).  Nonetheless,
as we shall see, the range that we can obtain is sufficient for
proving Theorem~\ref{thm1.1}.

As in \cite{HMSSZ}, due to technical difficulties in using the rotational
vector fields near $\partial\Omega$
(here $\partial_\theta = x_1\partial_2-x_2\partial_1$),
we shall modify the
Lebesgue spaces near $\partial\Omega$ from those in \eqref{15}.
Specifically,
given $0\le \gamma<1$, we define
\begin{equation}\label{3.2}
\|h\|_{X_{r,\gamma}}= \|h\|_{L^{s_\gamma}(|x|<3R)}+\|h\|_{L^r_{|x|}L^2_\theta(|x|>2R)},
\quad \text{with } \, \gamma=1-\frac 2{s_\gamma}\,.
\end{equation}
We fix $R\ge 1$ large enough so that $\Kob \subset \{|x|<R\}$.
When working with functions on $\R^2$,
the norms on the right side of \eqref{3.2} are
taken over $x\in {\mathbb R}^2$ with $|x|<3R$ and $|x|>2R$ for the
first and second terms, respectively.  For $\Omega$,
we define the norm in the obvious way
by extending $h$ to be equal to $0$ inside $\Kob$.

Note that $s_\gamma$ in \eqref{3.2} is chosen so that
$\dot H^\gamma({\mathbb R}^2)
\subset L^{s_\gamma}({\mathbb R}^2)$ and
$\dot H^\gamma(\Omega)\subset L^{s_\gamma}(\Omega)$, by Sobolev embedding.
We conclude by Lemma 2.2 of \cite{SmSo00} that
$$
\|u\|_{L^2_tL^{s_\gamma}_x({\mathbb R}_+ \times {\mathbb R}^2\,:\,|x|<3R)} \lesssim
\|f\|_{\dot H^\gamma({\mathbb R}^2)}+\|g\|_{\dot H^{\gamma-1}({\mathbb R}^2)}\, , \quad
0<\gamma \le \hf \, .
$$
Interpolating with energy conservation lets us conclude the same bound with
$2$ replaced by any $q\in[2,\infty]$.
By this and \eqref{15}, we conclude for the Minkowski
space case that, if $u$ solves \eqref{14} with forcing term $F\equiv 0$, and
$0< \gamma \le \frac 12$, then
\begin{multline}
\label{3.3}
\|u\|_{L^q_t X_{r,\gamma}({\mathbb R}_+ \times {\mathbb R}^2)} +
\|u\|_{L^\infty_t \dot H^\gamma({\mathbb R}_+\times {\mathbb R}^2)}
+ \|\partial_t u\|_{L^\infty_t \dot H^{\gamma-1}({\mathbb R}_+\times {\mathbb R}^2)}
\\
\lesssim \|f\|_{\dot H^\gamma({\mathbb R}^2)}+\|g\|_{\dot H^{\gamma-1}({\mathbb R}^2)},
\rule{0pt}{11pt}
\end{multline}
assuming that $q$, $r$ and $\gamma$ are as in \eqref{15}.

Using this estimate, the finite
propagation speed for $\square$, and the aforementioned Sobolev inequalities,
we see that
we also have a local in time variant of this estimate for $\Omega$.
Precisely, if $u$ solves the Dirichlet-wave equation \eqref{3.1}
with forcing term $F\equiv 0$, then
\begin{multline}
\label{3.4}
\|u\|_{L^q_t X_{r,\gamma}([0,1]\times \Omega)} + \|u\|_{L^\infty_t \dot H^\gamma([0,1]\times \Omega)}
+ \|\partial_t u\|_{L^\infty_t \dot H^{\gamma-1}([0,1]\times \Omega)}
\\
\lesssim \|f\|_{\dot H^\gamma(\Omega)}+\|g\|_{\dot H^{\gamma-1}(\Omega)},
\rule{0pt}{11pt}
\end{multline}
with the same assumptions on $q$, $r$ and $\gamma$.

We shall be able to use \eqref{3.3} and \eqref{3.4} to prove global variants
of some of the
estimates in \eqref{3.4} due to the fact that we have local energy decay
estimates for
the Dirichlet-wave equation \eqref{3.1}.  Specifically, given fixed $R_0>0$
we have
\begin{multline}\label{3.5}
\int_0^\infty \, \|u(t,\cd)\|^2_{H^1(|x|<R_0)} +
\|\partial_t u(t,\cd)\|_{L^2(|x|<R_0)}^2 \; dt
\\
\lesssim \|f\|^2_{H^1} + \|g\|_{L^2}^2 + \int_0^\infty \|F(s,\cd)\|^2_{L^2}\, ds\,,
\end{multline}
assuming that $\Kob$ is nonempty and nontrapping, and that
$f(x)$, $g(x)$ and $F(t,x)$
all vanish when $|x|>R_0$.  This was called ``Hypothesis 1.1" in \cite{HMSSZ}.
As noted there,
it follows from results of Vainberg \cite{V}, but another proof can be found
in Burq \cite{B}.  Also,
Ralston showed in \cite{R} that this estimate need not hold for
Neumann boundary conditions
in 2-dimensions, which explains why we are only treating the Dirichlet
case in this paper.

Since we have \eqref{3.3}--\eqref{3.5}, we can invoke
Theorem~1.4 from \cite{HMSSZ} to
conclude that we have global versions of \eqref{3.4} in the special
case where $\gamma=\hf$.
Precisely, if $u$ solves \eqref{3.1} with $F\equiv 0$, then
\begin{multline}\label{3.6}
\|u\|_{L^q_tX_{r, \hf}({\mathbb R}_+\times \Omega)} +
\|u\|_{L^\infty_t \dot H^{\hf}({\mathbb R}_+ \times \Omega)}
+ \|\partial_t u\|_{L^\infty_t \dot H^{-\hf}({\mathbb R}_+ \times \Omega)}
\\
\lesssim \|f\|_{\dot H^{\hf}(\Omega)}+ \|g\|_{\dot H^{-\hf}(\Omega)},
\end{multline}
assuming the following conditions on $q$ and $r$,
$$
q>2\,, \quad
\hf=1-\frac 2r-\frac 1q\,,
\quad \text{and} \quad
\frac 1q+\frac 1r<\hf\,.
$$

A limitation of Theorem~1.4 in \cite{HMSSZ}
(which seems difficult to overcome) is that for $n=2$
it applies only to the case of $\gamma=\hf$, whereas
for our existence proof we seek estimates with $0<\gamma<\hf$.
We get around this problem
by an interpolation argument.  Note that,
by Sobolev embedding and energy conservation,
if $0<\gamma<1$ and $s_\gamma$ is as in \eqref{3.2}, then
\begin{multline*}
\|u\|_{L^\infty_t L^{s_\gamma}_x({\mathbb R}_+\times \Omega)}
+ \|u\|_{L^\infty_t \dot H^\gamma ({\mathbb R}_+\times \Omega)}
+ \|\partial_t u\|_{L^\infty_t \dot H^{\gamma-1} ({\mathbb R}_+\times \Omega)}
\\
\lesssim \|f\|_{\dot H^\gamma(\Omega)}+\|g\|_{\dot H^{\gamma-1}(\Omega)}.
\end{multline*}
Since $s_\gamma\ge2$, it follows by H\"older's
inequality for $S^1$ that the $L^{s_\gamma}(\Omega)$ norm majorizes the
$X_{s_\gamma,\gamma}(\Omega)$ norm. Consequently, by the
preceding inequality we have that
\begin{multline}\label{3.7}
\|u\|_{L^\infty_t X_{s_\gamma,\gamma}({\mathbb R}_+\times \Omega)}
+\|u\|_{L^\infty_t \dot H^\gamma({\mathbb R}_+\times \Omega)}
+\|\partial_t u\|_{L^\infty_t \dot H^{\gamma-1}({\mathbb R}_+\times \Omega)}
\\
\lesssim \|f\|_{\dot H^\gamma(\Omega)}+\|g\|_{\dot H^{\gamma-1}(\Omega)},\quad
0<\gamma<1\,,
\end{multline}
if $u$ solves \eqref{3.1} with forcing term $F\equiv 0$.

\begin{figure*}[h]
\centering
\includegraphics[width=0.9\textwidth]{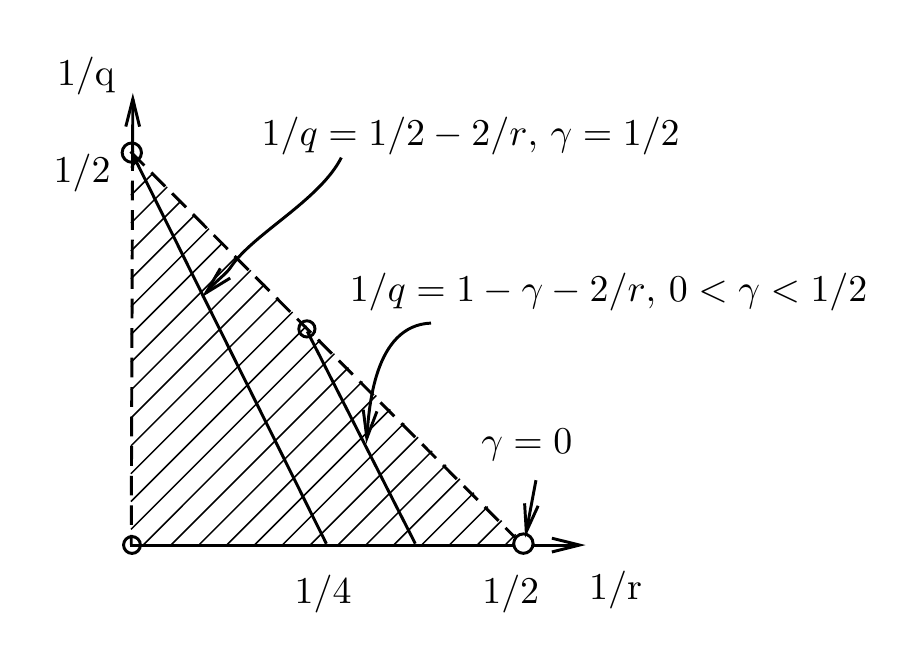}
\caption{Obstacle case exponents}\label{figue2}
\end{figure*}

In Figure \ref{figue2}, this corresponds to the exponents on the (open)
horizontal line segment corresponding to $\frac 1q=0$.
The global estimates \eqref{3.6} correspond to the
(half-open) segment where $\frac 1q=\hf-\frac 2r$ and $\gamma=\hf$
in this figure. Since
the convex hull of this line segment and the horizontal segment
is the shaded region in Figure \ref{figue2}, we conclude by interpolating
between \eqref{3.6} and \eqref{3.7} that,
for $u$ solving \eqref{3.1} with vanishing forcing term, we have
\begin{equation}\label{3.8}
\|u\|_{L^q_t X_{r,\gamma}({\mathbb R}_+\times \Omega)}
+\|u\|_{L^\infty_t \dot H^\gamma({\mathbb R}_+\times \Omega)}
+\|\partial_t u\|_{L^\infty_t \dot H^{\gamma-1}(\Omega)}
\lesssim \|f\|_{\dot H^\gamma(\Omega)}+\|g\|_{\dot H^{\gamma-1}(\Omega)},
\end{equation}
provided that
\begin{equation}\label{3.9}
q,r>2\,, \quad r<\infty\,, \quad
\gamma=1-\frac 2r-\frac 1q\,, \quad \text{and} \quad \frac 1q+\frac 1r<\hf\,.
\end{equation}
By the Christ-Kiselev lemma, if in addition $\tilde q$ and
$\tilde r$ satisfy the variant of \eqref{3.9} corresponding to $1-\gamma$,
\begin{equation}\label{3.10}
\tilde q, \tilde r>2\,, \quad \tilde r<\infty\,, \quad
1-\gamma = 1-\frac 2{\tilde r} - \frac 1{\tilde q}\,,
\quad \text{and} \quad
\frac 1{\tilde q} + \frac 1{\tilde r}<\hf\,,
\end{equation}
then if $u$ solves the linear Dirichlet-wave equation \eqref{3.1}
with forcing term $F$, we have
\begin{multline}\label{3.11}
\|u\|_{L^\infty_t \dot H^\gamma({\mathbb R}_+\times\Omega)}
+\|\partial_t u\|_{L^\infty_t \dot H^{\gamma-1}({\mathbb R}_+\times\Omega)}
+\|u\|_{L^q_t X_{r,\gamma}({\mathbb R}_+\times \Omega)}
\\
\lesssim \|f\|_{\dot H^\gamma(\Omega)}+\|g\|_{\dot H^{\gamma-1}(\Omega)}
+\|F\|_{L^{\tilde q'}_t X'_{\tilde r,1-\gamma}({\mathbb R}_+\times \Omega)}\,.
\end{multline}
Here, $X'_{\tilde r, 1-\gamma}$ denotes the norm which is dual to that of
$X_{\tilde r,1-\gamma}$. For the purposes of our existence proof
we do not need the exact expression for this dual norm, but use only
the following inequality. If $h=h_1+h_2$, and $h_1=0$ for $|x|>3R$, respectively
$h_2=0$ for $|x|<2R$, then
$$
\|h\|_{X'_{\tilde r,1-\gamma}}\le
\|h_1\|_{L^{s'_{1-\gamma}}(|x|<3R)}+
\|h_2\|_{L^{\tilde r'}_{|x|}L^2_\theta(|x|>2R)}\,,
$$
where $s'_{1-\gamma}$ and $\tilde r'$ denote the exponents which are conjugate to
$s_{1-\gamma}$ and $\tilde r$, respectively.
In particular, if $\phi$ and $\psi$ are smooth functions, with $\phi+\psi=1$,
and
$$
\phi(x)=\begin{cases}
1\,,\quad |x|<2R\,,\\ 0\,,\quad |x|>3R\,,
\end{cases}
$$
then
\begin{equation}\label{3.12}
\|h\|_{X'_{\tilde r,1-\gamma}}\le
\|\phi h\|_{L^{s'_{1-\gamma}}(|x|<3R)}+
\|\psi h\|_{L^{\tilde r'}_{|x|}L^2_\theta(|x|>2R)}\,.
\end{equation}

As with the proof of Glassey's theorem, we need a variant of \eqref{3.11}
involving the derivatives $\{\Gamma\}= \{\partial_t, Z\}$, where the $\{Z\}$ vector fields
are the ones in \eqref{1.4}. A problem arises
in establishing a version of \eqref{3.11} with derivatives, however,
in that the proof of such estimates on domains with boundary,
as in \cite{HMSSZ}, requires local energy decay estimates that
hold only for $\gamma=\frac 12$ in dimension $n=2$. Our approach will
be to establish estimates with derivatives for $\gamma=\frac 12$, and to interpolate
with \eqref{3.11} to obtain the desired estimates. For the case $\gamma=\hf$, we
have the following variant of Lemma 3.3 of \cite{HMSSZ}.

\begin{lemma}\label{lemma3.1}
Suppose that $(f,g,F)$ satisfy the Dirichlet compatibility conditions of order $k+\frac 12$.
Then, for even integers $k=0,2,4,\ldots$
\begin{multline}\label{3.18}
\sum_{|\alpha|\le k}\Bigl(\,
\|\Gamma^\alpha u\|_{L^\infty_t \dot H^{\hf}({\mathbb R}_+\times \Omega)}
+\|\Gamma^\alpha \partial_t u\|_{L^\infty_t \dot H^{-\hf}({\mathbb R}_+\times \Omega)}
+\|\Gamma^\alpha u\|_{L^q_t X_{r,\hf}({\mathbb R}_+\times\Omega)}
\, \Bigr)
\\
\lesssim
\sum_{|\alpha|\le k}\Bigl(\,
\|Z^\alpha \!f\|_{\dot H^{\hf}(\Omega)}+\|Z^\alpha \!g\|_{\dot H^{-\hf}(\Omega)}
+\|\Gamma^\alpha \! F\|_{L^{\tilde q'}_t X_{\tilde r,\hf}'({\mathbb R}_+\times\Omega)}
\, \Bigr),
\end{multline}
where $(q,r)$ and $(\tilde q,\tilde r)$ are as in \eqref{3.9} and \eqref{3.10}
for $\gamma=\hf$.
\end{lemma}
\noindent
\textbf{Remark.}
The condition on the data is that $\beta f\in H^{k+\frac 12}_D$ where $\beta$ is a compactly
supported cutoff to a neighborhood of the boundary, and similarly $\beta g\in H^{k-\frac 12}_D$.
The condition on $F$ is that $\beta F\in L^{\tilde q}_t H^{k-\frac 12}_D$. These
imply that for all $t$, $(\beta u(t,\cdot\,),\beta\partial_t u(t,\cdot\,)\in H^{k+\frac 12}_D\times H^{k-\frac 12}_D$, which will be used in elliptic regularity arguments.
We will use the fact that, if $f$ satisfies the $H^\hf_D$ boundary conditions,
then $\|f\|_{\dot H^\hf_D(\Omega)}\approx \|f\|_{\dot H^\hf(\Omega)}$, where the latter
is the norm for the space of restrictions of elements in $\dot H^\hf(\R^2)$.

\begin{proof}
The proof of Lemma \ref{lemma3.1} follows closely the proof of Lemma 3.3 of \cite{HMSSZ},
and we focus here on the modifications necessary for the above estimate.
Two key estimates needed for the proof are the following. Assuming that $F\equiv 0$ on $|x|\ge 3R$,
then
\begin{equation}\label{l2control}
\|u\|_{L^q_t X_{r,\hf}}+\|u\|_{L^\infty_t \dot{H}^{\hf}_D}
+\|\partial_t u\|_{L^\infty_t \dot{H}^{-\hf}_D}
\lesssim
\|f\|_{\dot{H}^{\hf}_D}+\|g\|_{\dot{H}^{-\hf}_D}+\|F\|_{L^2_tH^{-\hf}_D}\,,
\end{equation}
and, with no support assumptions on the data, the following holds
\begin{equation}\label{l2decay}
\|u\|_{L^\infty_t \dot{H}^{\hf}_D}+\|\partial_t u\|_{L^\infty_t \dot{H}^{-\hf}_D}+
\|\beta u\|_{L^2_t H^{\hf}_D}
\lesssim
\|f\|_{\dot{H}^{\hf}_D}+\|g\|_{\dot{H}^{-\hf}_D}+\|F\|_{L^{\tilde q}_t X_{\tilde r,\hf}'}\,.
\end{equation}
Estimate \eqref{l2control} follows from the estimates (2.8) of \cite{HMSSZ} and \eqref{3.11}, and
\eqref{l2decay} follows from \eqref{l2control} by duality.

Using \eqref{l2control} and \eqref{l2decay},
and the case $k=0$ of \eqref{3.18},
the argument on page 2803-2805 of \cite{HMSSZ} reduces estimate \eqref{3.18} to
bounding the following quantity by the right hand side of \eqref{3.18}, where
$\beta$ is a compactly supported cutoff to a neighborhood of the obstacle,
\begin{equation}\label{ulocbounds}
\sum_{j\le k}\|\beta \partial_t^j u\|_{L^2_t H^{\hf+k-j}_D}+
\sum_{j\le k+1}\|\beta \partial_t^j u\|_{L^\infty_t H^{\hf+k-j}_D}\,.
\end{equation}

We first observe that the Cauchy data
for $\partial_t^j u$, $j\le k\,,$ belongs to
$H^\hf_D\times H^{-\hf}_D$; this is seen by using the equation to express
$\bigl(\partial_t^j u(0,\cdot\,),\partial_t^{j+1} u(0,\cdot\,)\bigr)$ in terms of powers of
$\Delta$ applied to $(f,g,F)$, and observing that, by Sobolev embedding,
\begin{multline}\label{Finftybounds}
\sum_{|\alpha|\le l}\|\partial_{t,x}^\alpha F\|_{L^\infty_t\dot{H}^{\hf}}+
\sum_{|\alpha|\le l+1}\|\partial_{t,x}^\alpha F\|_{L^\infty_t\dot{H}^{-\hf}}+
\sum_{|\alpha|\le l}\|\partial_{t,x}^\alpha F\|_{L^2_t\dot{H}^{\hf}}
\\
\lesssim
\sum_{|\alpha|\le l+1}\|\partial_{t,x}^\alpha F\|_{L^{\tilde q}_t \dot H^{\hf}}+
\sum_{|\alpha|\le l+2}\|\partial_{t,x}^\alpha F\|_{L^{\tilde q}_t \dot H^{-\hf}}
\lesssim
\sum_{|\alpha|\le l+2}\|\partial_{t,x}^\alpha F\|_{L^{\tilde q}_t X'_{\hf,r}}\,.
\end{multline}
By \eqref{l2decay}, we thus conclude that
the following is bounded by the right hand side of \eqref{3.18}
\begin{equation}\label{utbounds}
\sum_{j\le k}\Bigl(\,\|\beta \partial_t^j u\|_{L^2_t H^\hf_D}+
\|\beta \partial_t^j u\|_{L^\infty_t H^\hf_D}+\|\beta \partial_t^{j+1} u\|_{L^\infty_t H^{-\hf}_D}
\Bigr)\,.
\end{equation}
To bound \eqref{ulocbounds}, it therefore suffices to bound the
following quantity by the right hand side of \eqref{3.18},
\begin{equation}\label{uderivs}
\sum_{l\le k/2}\Bigl(\,
\|\beta\partial_t^{k-2l}\Delta^l u\|_{L^2_tH^\hf_D}+
\|\beta\partial_t^{k-2l}\Delta^l u\|_{L^\infty_tH^\hf_D}+
\|\beta\partial_t^{k+1-2l}\Delta^l u\|_{L^\infty_tH^{-\hf}_D}\Bigr)\,.
\end{equation}
Here, we are using that we need consider only even powers of
$\partial_t$ in the first term of \eqref{ulocbounds} since $k$ is even, and since
\begin{equation*}
\|\beta \partial_t^j u\|^2_{L^2_t H^{\hf+k-j}_D}\le
\|\beta \partial_t^{j+1} u\|_{L^2_t H^{\hf+k-(j+1)}_D}\,
\|\beta \partial_t^{j-1} u\|_{L^2_t H^{\hf+k-(j-1)}_D}\,.
\end{equation*}
To bound \eqref{uderivs}, and conclude the proof,
we use the equation $(\partial_t^2-\Delta)u=F$ to express
$$
\partial_t^{k-2l}\Delta^lu=\partial_t^ku-\sum_{2j\le k-2}\partial_t^{k-2-2j}\Delta^j F\,.
$$
The resulting terms on the right may then be bounded in the appropriate norms using
\eqref{utbounds} and \eqref{Finftybounds}.
\end{proof}

The estimate that we shall require for the existence proof is the
following. It is valid provided that $\phi\,,\psi\,,\tilde\phi\,,\tilde\phi\in
C^\infty(\R^2)$
take values in $[0,1]$, with
\begin{equation}
\text{supp}(\phi),\,\text{supp}(\tilde\phi)\subset\{|x|<3R\}\,,\quad
\text{supp}(\psi),\,\text{supp}(\tilde\psi)\subset\{|x|>2R\}\,,\quad
\tilde\phi+\tilde\psi\ge 1\,.
\qquad
\end{equation}
We will additionally assume
that each is a radial function, and that $\phi=1$ (respectively $\psi=1$)
on a neighborhood of the support of $\tilde\phi$ (respectively $\tilde\psi$).

\begin{corollary}\label{coro-3.2}
Suppose that $\gamma$, $(q,r)$ and $(\tilde q,\tilde r)$ are as in \eqref{3.9}
and \eqref{3.10}. Suppose also that
$F$ satisfies the Dirichlet compatibility conditions of order $1+\gamma$. Then
for the solutions $u$ of \eqref{3.1} with vanishing Cauchy data,
\begin{multline}\label{3.17}
\sum_{|\alpha|\le 1}\Bigl(\,
\|Z^\alpha u\|_{L^\infty_t L^{s_\gamma}_x}
+\|\psi\,\Gamma^\alpha u\|_{L^q_tL^r_{|x|}L^2_\theta}
+\|\phi\,\Gamma^\alpha u\|_{L^q_tL^{s_\gamma}_x}
\, \Bigr)
\\
\lesssim
\sum_{|\alpha|\le 1}
\Bigl(\,
\|\tilde\psi\,\Gamma^\alpha F\|_{L^{\tilde q}_tL^{\tilde r}_{|x|}L^2_\theta}
+\|\tilde\phi\,\Gamma^\alpha F\|_{L^{\tilde q}_tL^{s_{1-\gamma}'}_x}
\, \Bigr)\,,
\end{multline}
where all norms are taken over ${\mathbb R}_+\times \Omega$.
\end{corollary}
\begin{proof}
Estimate \eqref{3.17} is obtained by interpolating estimate \eqref{3.18},
which requires $\gamma=\hf$ but allows arbitrarily high order powers of $\Gamma$, with
estimate \eqref{3.11}, which holds for all $0<\gamma<1$, but with 0 powers
of $\Gamma$. We thus need to justify the interpolation step by expressing the norms
in terms of analytic scales of spaces.
We start by noting that
$$
\sum_{|\alpha|\le 1}\|\phi\,\Gamma^\alpha u\|_{L^q_tL^{s_\gamma}_x}
\approx
\sum_{|\alpha|\le 1}\|\phi\,\partial_{t,x}^\alpha \,u\|_{L^q_tL^{s_\gamma}_x}\,,
$$
and
$$
\sum_{|\alpha|\le 1}\|\psi\,\Gamma^\alpha u\|_{L^q_tL^r_{|x|}L^2_\theta}
\approx
\sum_{|\alpha|\le 1}\|\psi\,\partial_{t,r,\theta}^\alpha \,u\|_{L^q_tL^r_{|x|}L^2_\theta}\,,
$$
with similar equalities for the norms in $F$. By embedding $\Omega\cap\{|x|<3R\}$ in a
compact manifold with boundary, the first norm is dominated by (with a different choice of
$\phi$)
$$
\sum_{|\alpha|\le 1}\|D^\alpha(\phi u)\|_{L^q_tL^{s_\gamma}_x(\R\times\Omega')}
$$
which is a Sobolev norm on a mixed-norm space. That the fractional order Sobolev norms
$$
\|(1- \partial_t^2-\Delta_x)^{\sigma/2} u\|_{L^q_tL^{s_\gamma}_x(\R\times\Omega')}
$$
form an analytic scale of spaces, and that norms for integer $\sigma$ coincide with
partial derivatives of order up to $\sigma$ belonging to the mixed-norm space, follows
from the fact that Calder\'on-Zygmund operators are bounded in mixed-norm $L^p$ spaces,
provided that all Lebesgue exponents lie in the range $(1,\infty)$. See Lizorkin
\cite{Liz} for the case of $\R^n$. The product manifold setting falls under the theory
of UMD spaces; see, for example, \cite{Zim}.

The norms over $\{|x|>2R\}$ are similarly product norms over polar coordinates.
Precisely,
$$
\sum_{|\alpha|\le 1}\|\psi\,u\|_{L^q_tL^r_{|x|}L^2_\theta}\approx
\biggl(\,\int_\R\,\biggl(\,\int_\R\;\biggl(\,\int_{S^1}|\psi(\rho)u(t,\rho,\theta)|^2
\,d\theta\,\biggr)^{r/2}\langle\rho\rangle\,d\rho\,\biggr)^{q/r}dt\biggr)^{1/q}\,.
$$

If $W$ denotes the forward solution operator to the wave equation on $\Omega$,
in that $u=WF$, then \eqref{3.17} can be stated in terms of mapping properties
of $\phi W\!\tilde\phi$, $\phi W\!\tilde\psi$, $\psi W\!\tilde\phi$ and $\psi W\!\tilde\psi$
between such spaces, where the cutoffs $\phi$, etc., may vary from above.
For example, we need the bound, for $k=1$,
\begin{equation}\label{nearbound}
\sum_{|\alpha|\le k}\|D^\alpha(\phi W\!\tilde\phi F)\|_{L^q_tL^{s_\gamma}_x(\R\times\Omega')}
\lesssim
\sum_{|\alpha|\le k}\|D^\alpha F\|_{L^{\tilde q}_tL^{s_{1-\gamma}'}_x(\R\times\Omega')}\,,
\end{equation}
where we may think of $F$ as a function of $(t,x)\in \R\times\Omega'\,.$
By \eqref{3.12} and Lemma \ref{lemma3.1}, this holds for all even integers
$k$ provided $\gamma=\hf$, and by \eqref{3.11} it holds for $k=0$ for all $0<\gamma<1$.
Since the relations \eqref{3.9} and \eqref{3.10} are linear in the reciprocals of
$s_\gamma$, $q$, and $r$ (respectively $s'_{1-\gamma}$, $\tilde q$, and $\tilde r$),
we may interpolate to obtain \eqref{nearbound} for $k=1$ at any point in the shaded
region. The estimates for the other terms follow similarly, using the embedding
$\dot H^\gamma\subset L^{s_\gamma}$ for the term $\|u\|_{L^\infty_t L^{s_\gamma}}$.

We note that if $.2<\gamma<.5$, as in our application, then it suffices to
consider $k\le 4$ for the estimate \eqref{3.18}, since one may take
the other endpoint with $k=0$ arbitrarily close to the lower right corner.
\end{proof}

Corollary \ref{coro-3.2} gives us the required estimates for the inhomogeneous equation,
but we also need the following estimates for the homogeneous wave equation.
\begin{corollary}\label{coro-3.3}
Suppose that $\gamma$ and $(q,r)$ are as in \eqref{3.9}.
Suppose also that $f|_{\partial \Omega}=0$.
Then
for the solutions $u$ of \eqref{3.1} with $F\equiv 0$,
\begin{multline}\label{3.17'}
\sum_{|\alpha|\le 1}\Bigl(\,
\|Z^\alpha u\|_{L^\infty_t L^{s_\gamma}_x}
+\|\psi\,\Gamma^\alpha u\|_{L^q_tL^r_{|x|}L^2_\theta}
+\|\phi\,\Gamma^\alpha u\|_{L^q_tL^{s_\gamma}_x}
\, \Bigr)
\\
\lesssim
\sum_{|\alpha|\le 2}
\|Z^\alpha f\|_{L^{s'_{1-\gamma}}(\Omega)}+\sum_{|\alpha|\le 1}\|Z^\alpha g\|_{L^{s'_{1-\gamma}}(\Omega)}\,.
\end{multline}
\end{corollary}

This is the one step in the proof of Theorem \ref{thm1.1} where condition \eqref{1.5}
is used. Indeed, one can replace \eqref{1.5} by any norm condition which implies that
the left hand side of \eqref{3.17'} is sufficiently small (where $\gamma=\gamma_p$).
The norms we are using for the initial data are stronger than the norms in \eqref{19} using inhomogeneous Sobolev spaces, since
\begin{equation}\label{sobolev}
W^{1,s'_{1-\gamma}}\subset \dot H^\gamma\,, \qquad L^{s'_{1-\gamma}}\subset \dot H^{\gamma-1}\ .
\end{equation}
Note also that $s'_{1-\gamma_p}=q_p$, where $q_p$ is as in \eqref{1.5}.
We use the above norm due to the difficulty in showing that
$\sum_{|\alpha|\le k}\|Z^\alpha f\|_{\dot H^\gamma(\Omega)}$ defines an interpolation
scale of spaces, simultaneously in $k$ and $\gamma$.

Because of \eqref{sobolev} (see \eqref{sob}), we immediately find that when $F\equiv 0$, \eqref{3.8} and \eqref{3.18} respectively imply the somewhat weaker versions
\begin{equation}\label{aa}
\|u\|_{L^\infty_tL^{s_\gamma}}+
\|u\|_{L^q_tX_{r,\gamma}}\lesssim \sum_{|\alpha|\le1}\|Z^\alpha f\|_{L^{s'_{1-\gamma}}}
+\|g\|_{L^{s'_{1-\gamma}}},\end{equation}
and, for $k=0,2,4,\dots$ and $(f,g)$ satisfying the compatibility conditions of order
$k+\tfrac12$,
\begin{equation}\label{bb}\sum_{|\alpha|\le k}
\bigl(\, \|\Gamma^\alpha u\|_{L^\infty_tL^4}
+
\|\Gamma^\alpha u\|_{L^q_t X_{r,\frac 12}}
\, \bigr)
\lesssim \sum_{|\alpha|\le k+1}\|Z^\alpha f\|_{L^{\frac 43}}+\sum_{|\alpha|\le k}\|Z^\alpha
g\|_{L^{\frac 43}}\,.
\end{equation}
These are the inequalities that we use in the interpolation argument
to get \eqref{3.17'}.

The interpolation arguments are similar to those used to prove
the inhomogeneous estimate \eqref{3.17}.
For example, for $f$ we have
$$\sum_{|\alpha|\le 2}\|Z^\alpha f\|_{L^{{s'_{1-\gamma}}}(\Omega)}
\approx \sum_{|\alpha|\le 2}\|\tilde\psi Z^\alpha f\|_{L^{{s'_{1-\gamma}}}(\Omega)}
 +\sum_{|\alpha|\le 2}\|\tilde\phi Z^\alpha f\|_{L^{{s'_{1-\gamma}}}(\Omega)}\ ,
$$
and
\begin{align*}
\sum_{|\alpha|\le 2}\|\tilde\psi\, Z^\alpha f\|_{L^{{s'_{1-\gamma}}}(\Omega)}
&\approx  \sum_{|\alpha|\le 2}\|\tilde\psi\, \partial_{r,\theta}^\alpha f\|_{L^{{s'_{1-\gamma}}}(\Omega)}\,,\\
\sum_{|\alpha|\le 2}\|\tilde\phi \,Z^\alpha f\|_{L^{{s'_{1-\gamma}}}(\Omega)}
&\approx  \sum_{|\alpha|\le 2}\|\tilde\phi\,\partial_x^\alpha f\|_{L^{{s'_{1-\gamma}}}(\Omega)}\,.\rule{0pt}{18pt}
\end{align*}
The term $\sum_{|\alpha|\le 2}\|\tilde\phi\,\partial_x^\alpha f\|_{L^{{s'_{1-\gamma}}}(\Omega)}$, can be bounded from above and below by
$$
\|\tilde\phi f\|_{W^{2,{s'_{1-\gamma}}}}
=\sum_{|\alpha|\le 2} \|\partial_x^\alpha (\tilde\phi f)\|_{L^{{s'_{1-\gamma}}}}
$$
(with different choices of $\tilde\phi$), which is a standard Sobolev space norm.

If $U$ denotes the solution operator to the wave equation \eqref{3.1} on $\Omega$, with $F, g\equiv 0$,
then \eqref{3.17'}
for this special case
can be restated in terms of mapping properties
of $\phi U\!\tilde\phi$, $\phi U\!\tilde\psi$, $\psi U\!\tilde\phi$ and $\psi U\!\tilde\psi$
between these spaces.
For example, we need the bound, for $k=1$,
\begin{equation}\label{nearbound'}
\sum_{|\alpha|\le k}\|\partial_{t,x}^\alpha(\phi U\!\tilde\phi f)\|_{L^q_t L^{s_\gamma}_x (\R\times\Omega')}
\lesssim
\sum_{|\alpha|\le k+1}\|\partial_x^\alpha f\|_{L^{s'_{1-\gamma}}_x(\Omega')}\,,
\end{equation}
where we may think of $f$ as a function of $x\in \Omega'\,.$
By
\eqref{bb}
this holds for all even integers
$k$ provided $\gamma=\hf$, and by
\eqref{aa}
it holds for $k=0$ for all $0<\gamma<1$.
Since the relation \eqref{3.9} is linear in the reciprocals of
$s_\gamma$, $q$, and $r$,
we may interpolate to obtain \eqref{nearbound'} for $k=1$ at any point in the shaded
region
of Figure~2.

Therefore, we conclude the estimate \eqref{3.17'} is valid in the special case
where $g\equiv 0$.
Since similar arguments apply to the case where $f\equiv 0$, we get \eqref{3.17'}.


\medskip

We shall now show how we can use \eqref{3.17} and \eqref{3.17'} to prove our existence results.

\medskip

\noindent {\bf Proof of Theorem~\ref{thm1.1}:}
As in our proof of Glassey's theorem, it suffices to consider the case
of $p_c<p<5$. The proof in the obstacle case requires more care in selecting
the indices $q$ and $r$, since the case $r=1$ is not allowed in \eqref{3.17},
as opposed to its free-space variant \eqref{15}.
We thus need to
check that we can choose exponents whose ratio is $p$, so that we have estimates
that iterate well for equations like $\square u = |u|^p$.

To do this, assume given $p$ such that $p_c<p<5$. We will take
$\gamma=\gamma_p$ to be the scaling index for $\square u=|u|^p$,
$$
\gamma_p=1-\frac2{p-1}\,,
$$
so that $\frac {5-\sqrt{17}}4<\gamma_p<\hf\,.$ 
As noted before,
the condition $\hf-\gamma_p<\frac 1p$ is equivalent to the condition
$p>p_c$. This tells us that
\begin{equation}\label{3.15}
\frac 12-\gamma_p<\frac 1{r} \quad \text{if}\quad p< r<p+\delta(p)\,,
\end{equation}
for small $\delta(p)>0$. We fix such an $r$, and determine
$\tilde r$ by setting ${\tilde r}'=r/p$.
Since we may assume $\delta(p)<p$, then $p<r<2p$, so that
$\tilde r\in (2,\infty)$.

The equality in conditions \eqref{3.9} and \eqref{3.10} determines that,
with $\gamma=\gamma_p$,
\begin{equation*}
q(\gamma_p,r)=\frac{p-1}2 \, \cdot \frac{r}{r-(p-1)}\,,
\qquad
[\tilde q(1-\gamma_p,\tilde r)]' = \frac{p-1}2 \cdot \frac{\tilde
r'}{p\tilde r'-(p-1)}\,.
\end{equation*}
Since ${\tilde r}'=r/p$, we have
\begin{equation}\label{3.15'}
[\tilde q(1-\gamma_p, \tilde r)]' = q(\gamma_p,r)/p \,.
\end{equation}

The last inequalities in \eqref{3.9} and \eqref{3.10}
are then equivalent to the conditions
\begin{equation*}
\frac12 - \gamma_p < \frac1r \quad \text{and} \quad
\gamma_p-\frac12 < \frac1{\tilde r}\,.
\end{equation*}
The first condition
is satisfied by \eqref{3.15}, and the second is satisfied since
$\gamma_p < \hf$.

To conclude the verification of \eqref{3.9}-\eqref{3.10}, we check that
$2<q,r,\tilde q,\tilde r<\infty$. By construction this holds for $r,\tilde r$.
We next observe that
$q(\gamma_p,r)$ is a decreasing function of $r$ for $r>p$, and
$$
p+1<q(\gamma_p,p)<2p\,.
$$
The first inequality here is equivalent to $p^2-3p-2>0$, and the second
to $p<5$. Taking $\delta(p)$ smaller if necessary,
it follows that $q(\gamma_p,r)\in (p,2p)\subset(2,\infty)$,
and hence $q(1-\gamma_p,\tilde r)\in (2,\infty)$ by \eqref{3.15'}.

With this choice of indices, we then have the following case of \eqref{3.17},
valid for solutions $u$ with vanishing Cauchy data:
\begin{multline}\label{3.16}
\sum_{|\alpha|\le 1}
\Bigl(\,
\|\psi\,\Gamma^\alpha u\|_{L^q_tL^r_{|x|}L^2_\theta}
+\|\phi\,\Gamma^\alpha u\|_{L^q_tL^{s_{\gamma_p}}_x}
+\|Z^\alpha u_k\|_{L^\infty_t L^{s_{\gamma_p}}_x}
\, \Bigr)\quad
\\
\lesssim
\sum_{|\alpha|\le 1}
\Bigl(\,
\|\tilde\psi\,\Gamma^\alpha \square u\|_{L^{q/p}_tL^{r/p}_{|x|}L^2_\theta}
+\|\tilde\phi\,\Gamma^\alpha \square u\|_{L^{q/p}_tL^{s_{1-\gamma_p}'}_x}
\,\Bigr)
\,.
\end{multline}

We now assume that the Cauchy data $(f,g)$
satisfies the smallness condition \eqref{1.5} (where $q_p$ is the same as our notation $s'_{1-\gamma_p}$),
and let $u_0$ solve the Cauchy problem \eqref{3.1}
with forcing term
$F\equiv0$.  We then iteratively define $u_k$, $k=1,2,3,\dots$,
by requiring that it solves
the equation
\begin{equation*}
\begin{cases}
\square u_k(t,x)=F_p(u_{k-1}(t,x)), \quad (t,x)\in {\mathbb R}_+\times \Omega
\\
u_k(t,x)=0, \quad (t,x)\in {\mathbb R}_+\times \partial\Omega
\\
u_k|_{t=0}=f, \quad \partial_t u_k|_{t=0}=g.
\end{cases}
\end{equation*}
Our goal is to show that if $\varepsilon>0$ in \eqref{1.5} is small enough
then so is
\begin{equation*}
M_k=\sum_{|\alpha|\le 1}\Bigl( \,
\|\psi\,\Gamma^\alpha u_k\|_{L^q_tL^r_{|x|}L^2_\theta }
+\|\phi\,\Gamma^\alpha u_k\|_{L^q_tL^{s_{\gamma_p}}_x}
+\|Z^\alpha u_k\|_{L^\infty_t L^{s_{\gamma_p}}_x}
 \, \Bigr)
\end{equation*}
for every $k=0,1,2,\dots$, where we fix $r$
and $q=q(\gamma_p,r)$ as in \eqref{3.17}.

For $k=0$, it follows from \eqref{3.17'} that
$M_0\le C_0 \varepsilon$, with $C_0>1$ a fixed constant.  For $k=1,2,\dots$, we can then use \eqref{3.17} and \eqref{3.17'} to conclude
that
\begin{multline}\label{3.19}
M_k\le C_0\varepsilon+C_1\!\!\sum_{|\alpha|\le1} \Bigl(\|\tilde\psi\,\Gamma^\alpha
F_p(u_{k-1})\|_{L^{q/p}_tL^{r/p}_{|x|}L^2_\theta(|x|>2R)}+
\|\tilde\phi\,\Gamma^\alpha
F_p(u_{k-1})\|_{L^{q/p}_tL^{s_{1-\gamma_p}'}_x(|x|<3R)}\,\Bigr)
\\
=C_0\varepsilon + C_1 (I + II)\,,
\end{multline}
with $C_1$ another fixed constant.
Assuming that $M_{k-1}\le 2 C_0 \,\ep$, we will inductively show that
$M_k\le 2 C_0 \,\ep$.

We first note that since $s_{\gamma_p}=p-1>2$ and $n=2$, it follows from
Sobolev embedding on $\Omega$ that
$$
\|v\|_{L^\infty}\lesssim
\sum_{|\alpha|\le 1}\|\partial^\alpha v\|_{L^{s_{\gamma_p}}}
\ .
$$
This means that
$$
\|u_{k-1}(t,x)\|_{L^\infty_t L^\infty_x}\le C
M_{k-1}\le 2 C C_0 \,\ep\le 1\,,
$$
provided that $\ep$ is small enough,
which verifies the condition on $u$ in \eqref{1.2}. Our assumption
\eqref{1.2} on the nonlinear term, $F_p$, then implies that
\begin{equation}
\label{inhom}\sum_{|\alpha|\le1}|\Gamma^\alpha
F_p(u_{k-1})|\lesssim |u_{k-1}|^{p-1}\sum_{|\alpha|\le 1}
|\Gamma^\alpha u_{k-1}|\ .
\end{equation}

Since the collection $\{\Gamma\}$ contains $\partial_\theta$,
by Sobolev embedding on the circle we have
$$
\|v(|x|\, \cdot\, )\|_{L^\infty_\theta}\lesssim \sum_{|\alpha|\le 1}
\|\Gamma^\alpha v(|x|\, \cdot\, )\|_{L^2_\theta}\ , \quad |x|>2R\ .
$$
Consequently, since $\Omega$ contains the set $|x|>2R$,
it follows for fixed $|x|>2R$ and $t>0$ that
$$
\sum_{|\alpha|\le1}\|\Gamma^\alpha F_p(u_{k-1}(t,|x|\, \cdot \, ))\|_{L^2_\theta}
\lesssim \sum_{|\alpha|\le 1}\|\Gamma^\alpha u_{k-1}(t,|x|\, \cdot
\, )\|_{L^2_\theta}^p\ ,
$$
which means that $I\le C_2 M^p_{k-1}$, for some
uniform constant $C_2$.

To handle the term $II$ in \eqref{3.19}, we note that since
$s_{\gamma_p}>2$ and $n=2$, it follows from Sobolev embedding
on $\Omega \cap \{|x|<3 R\}$ that
$$
\|\tilde \phi v\|_{L^{2(p-1)} (|x|<3 R)}\lesssim
\sum_{|\alpha|\le 1}\|\phi Z^\alpha v\|_{L^{s_{\gamma_p}}(|x|<3 R)}
\,.
$$
Since $s_{1-\gamma}'<2$ satisfies
$$
\frac{1}{s_{1-\gamma}'}=\frac{1}{2}+\frac{1}{s_{\gamma}}\,,
$$
by H\"older's inequality we have, for each fixed $t$, that
\begin{multline*}\sum_{|\alpha|\le1}
\|\tilde\phi\,\Gamma^\alpha F_p(u_{k-1}(t,\cd))\|_{L^{s_{1-\gamma}'}(|x|<3R)}
\\
\lesssim \sum_{|\alpha|\le1}
\|\tilde\phi \,u_{k-1}\|_{L^{2(p-1)}(|x|<3R)}^{p-1}
\sum_{|\alpha|\le1}\|\phi\,\Gamma^\alpha
u_{k-1}(t,\cd)\|_{L^{s_{\gamma}}(|x|<3R)}
\\
\lesssim \sum_{|\alpha|\le1} \|\phi\,\Gamma^\alpha
u_{k-1}(t,\cd)\|_{L^{s_{\gamma}}(|x|<3R)}^p\,.
\end{multline*}
This implies that we also have $II\le C_3 M_{k-1}^p$, for some
uniform constant $C_3$, which together with the bound for $I$ gives
$$M_k\le C_0\,\varepsilon +C_1(C_2+C_3)M^p_{k-1}\le C_0\,\varepsilon +
C_1(C_2+C_3)(2 C_0 \,\ep)^p.$$
Thus, if $\varepsilon$ is sufficiently small,
we conclude that
\begin{equation}\label{3.20}
M_k\le 2C_0\,\varepsilon, \quad k=0,1,2,\cdots
\end{equation}

To finish the proof of the existence results, for $p_c<p<5$
we need to show that the
$u_k$ converge to a solution of \eqref{1.1}.
To do
this it suffices to show that $$A_k= \bigl\| \, u_k-u_{k-1}\,
\bigr\|_{L^{\infty}_t \dot H^{\gamma_p}}$$ tends geometrically to zero as $k\to
\infty$.  Since $|F_p(v)-F_p(w)|\lesssim |v-w|(\,
|v|^{p-1}+|w|^{p-1}\, )$ when $v$ and $w$ are small, the proof of
\eqref{3.20} can be adapted to show that, for small $\varepsilon>0$,
there is a uniform constant $C$ so that
$$A_k\le CA_{k-1}(M_{k-1}+M_{k-2})^{p-1},$$
which, by \eqref{3.20}, implies that $A_k\le \tfrac12A_{k-1}$ for
small $\varepsilon$.  Since $A_1$ is finite, the claim follows,
which finishes the proof of the existence results for the range of
$p_c<p<5$.

As in \S 2, the results for $p\ge5$ in Theorem~\ref{thm1.1}
follow from the above and the fact that the condition
\eqref{1.2} becomes weaker as $p$ increases. \qed

\end{document}